\documentclass[reqno,12pt]{amsart}

\usepackage[utf8]{inputenc}
\usepackage[T1]{fontenc}

\usepackage{fourier}
\usepackage{amssymb, amsmath, amsthm}

\usepackage{mathtools}
\usepackage{esint}

\newtheorem{defi}{Definition}
\newtheorem{theorem}[defi]{Theorem}

\newtheorem{rem}[defi]{Remark}
\newtheorem{prop}[defi]{Proposition}

\newcommand{\R}{\mathbb{R}}

\newcommand{\eps}{\varepsilon}
\newcommand{\ve}{\varepsilon}

\newcommand{\vf}{\varphi}

\newcommand{\nd}{\partial_\nu}

\usepackage[dvipsnames]{xcolor}

\newcommand{\dtn}{\mathcal{D}}
\def\<{{\langle}}
\def\>{{\rangle}}

\DeclareMathOperator*{\grad}{grad}
\DeclareMathOperator*{\dt}{dt}

\DeclareMathOperator*{\St}{S}
\DeclareMathOperator*{\Ne}{N}
\DeclareMathOperator*{\Vol}{Vol}

\usepackage{setspace}

\usepackage{marginnote}
\let\oldmarginnote\marginnote
\renewcommand{\marginnote}[1]
{\oldmarginnote{
\vspace{-.3cm}\blue{\tiny
\begin{spacing}{1}#1
\end{spacing}}}
}

\title[Large Steklov eigenvalues]{Large Steklov eigenvalues under volume constraints}

\author{Alexandre Girouard}
\address{D\'epartement de math\'ematiques et de statistique, Pavillon Alexandre-Vachon, Universit\'e Laval, Qu\'ebec, QC, G1V 0A6, Canada}
\email{alexandre.girouard@mat.ulaval.ca}

\author{Panagiotis Polymerakis}
\address{Department of Mathematics, University of Thessaly,
3rd km Old National Road Lamia – Athens, 35100, Lamia, Greece.}
\email{ppolymerakis@uth.gr}

\date{\today}

\begin{document}
\begin{abstract}
In this note we establish an expression for the Steklov spectrum of warped products in terms of auxiliary Steklov problems for drift Laplacians with weight induced by the warping factor.
As an application, we show that a compact manifold with connected boundary diffeomorphic to a product admits a family of Riemannian metrics which coincide on the boundary, have fixed volume and arbitrarily large first non-zero Steklov eigenvalue. These are the first examples of Riemannian metrics with these properties on three-dimensional manifolds. 
\end{abstract}

\maketitle

\section{\bf Introduction}
Let $(M,g_M)$ be a smooth connected compact Riemannian manifold of dimension $m\geq 2$ with boundary $\Sigma=\partial M$, and consider the Laplace operator $\Delta$ be acting on smooth functions on $M$.
The Dirichlet-to-Neumann map $\dtn_\Delta:C^\infty(\Sigma)\to C^\infty(\Sigma)$ associated to $\Delta$ is defined by
$\dtn_\Delta f=\nd{\bar{f}}$, where $\nu$ is the outward normal  along the boundary $\Sigma$ and where the function $\bar{f}\in C^\infty(M)$ is the unique harmonic extension of $f$ to the interior of $M$.
Because it is an elliptic pseudodifferential operator, $\dtn_\Delta$ is essentially self-adjoint. Its closure (which we also denote $\dtn_\Delta$) has non-negative, unbounded discrete spectrum $\sigma_{\dtn_\Delta}(M)$:
  $$0=\sigma_0(M,g_M)<\sigma_1(M,g_M)\leq\sigma_2(M,g_M)\leq\cdots\to\infty,$$
  where each eigenvalue is repeated according to its multiplicity.
  The eigenvalues of $\dtn_{\Delta}$ are known as \emph{Steklov eigenvalues} of $M$. Their interplay with the geometry of $M$ is an active area of investigation. See \cite{GiPo2017,CoGiGoShe2022} for surveys.

The general question that motivates this paper is to know how large the spectral gap $\sigma_1(M,g)>0$ can be for metrics $g$ that are prescribed on the boundary $\partial M$. More precisely, let $g_M$ be the Riemannian metric of $M$. The question is to compute, or estimate, the following quantity:
  \begin{gather}\label{eq:sup}
    \sup\{\sigma_1(M,g)\,:\,g=g_M\text{ on }\partial M\}.
  \end{gather}

For surfaces, this question has been investigated for a long time, starting with the work of Weinstock~\cite{We1954} who considered simply-connected planar domains, and subsequently by many authors, leading in particular to the following upper bound by Kokarev~\cite{Ko2014}, in terms of the genus $\gamma$ of an orientable compact surface $M$ and the length $L_g(\partial M)$ of its boundary:
  \begin{gather}\label{ineq:Kokarev}
    \sigma_1(M,g)L_g(\partial M)\leq 8\pi(\gamma+1).
  \end{gather}
  For surfaces, the determination of the exact value of the supremum~\eqref{eq:sup} has been the subject of intense activity over the last ten years. See~\cite[Section 3]{CoGiGoShe2022} for the state of the art, circa 2022.

The situation for surfaces is atypical. 
  Indeed, on any compact manifold $M$ of dimension at least 3, there exists a family  of positive smooth functions $(\omega_\eps)_{\eps>0}\subset C^{\infty}(M)$ such that $\omega_\eps\equiv 1$ on $\partial M$ and the conformal metrics $g_\eps=\omega_{\eps}g_M$ satisfy
  $$\sigma_1(M,g_\eps)\xrightarrow{\eps\to0}+\infty.$$
  This was proved by Colbois, El Soufi and Girouard~\cite{CoElGi2019}.  It follows from work of Fraser and Schoen~\cite{FrSc2011} that any such family $g_\eps$ must satisfy
  $$\text{Vol}(M,g_\eps)\xrightarrow{\eps\to 0}+\infty.$$
  Indeed, they proved that any conformal metric $g\in[g_M]$ satisfies
  $$\sigma_1(M,g)\text{Vol}_g(\partial M)\leq C([g_M])\text{Vol}_g(M,g)^{\frac{n-2}{n}},$$
  where $C([g_M])$ is a finite conformal invariant, which they call the \emph{relative conformal volume}, in reference to the classical conformal volume of Li and Yau~\cite{LiYa1982}.  See also the work of Hassannezhad~\cite{Ha2011} for eigenvalues $\sigma_k$ of arbitrary index.
  This raises the question of how large $\sigma_1$ could be among all metrics $g$ of unit volume that coincide with $g_M$ on the boundary $\partial M$, without the conformal constraint. In other words, the question is to compute
  $$S(M,g_M):=\sup\{\sigma_1(M,g)\,:\,g=g_M\text{ on }\partial M\,\text{ and }\,\Vol(M,g)=1\}.$$
  \begin{rem}
    The requirement that $g=g_M$ on the boundary of $M$ is restrictive. Indeed, recall that any closed manifold $\Sigma$ of dimension $d\geq 3$ admits Riemannian metrics $(g_\eps)_{\eps>0}$ of unit volume such that the first non-zero eigenvalue of the Laplace operator satisfies $\lambda_1(\Sigma,g_\eps)\xrightarrow{\eps\to 0}+\infty$. This was proved by Colbois and Dodziuk in~\cite{CoDo1994}. It follows from separation of variables, that $M=[0,1]\times \Sigma$ equipped with the direct sum metric $dt^2+g_\eps$ satisfies
    $\sigma_1(M,g_\eps)\xrightarrow{\eps\to0}+\infty$. These metrics have unit volume, however they are not prescribed on the boundary of $M$.
  \end{rem}  
  For manifolds $M$ of dimension at least 4 with boundary $\partial M$ that admits a Killing vector field whose dual 1-form $\xi$ has nowhere vanishing exterior derivative $d\xi$, Cianci and Girouard~\cite{MR3878841} constructed a family of metrics $g_\eps$ of unit volume which satisfy $\sigma_1(M,g_\eps)\xrightarrow{\eps\to0}+\infty$. In other words, these manifolds satisfy $S(M,g_M)=+\infty$. The proof is inspired by work of Bleecker~\cite{MR678360}, establishing the existence of Riemannian metrics with fixed volume and arbitrarily large spectral gap on any closed manifold of dimension at least three, admitting a Killing vector field satisfying the above condition.
Because the construction in~\cite{MR3878841} is adapted from that of Bleecker on the boundary of the manifold, the dimension of the boundary needs to be at least 3 in this construction.

The purpose of this paper is to present two new families of examples where $S(M,g)=+\infty$. Our method provides the first three-dimensional examples, but also higher-dimensional examples that are different from those of~\cite{MR3878841}.
  In our first result, not only the volume $\text{Vol}(M,g_M)$ is preserved, but the volume element $dV_{g_M}$ is also preserved.
\begin{theorem}\label{thm:product}
Let $(M,g_M)$ be the Riemannian product of a compact manifold $B^n$ with boundary and a closed manifold $F^k$, where $n > k\geq 1$. Then for any $\ve > 0$ there exists a Riemannian metric $g_\ve$ on $M$ that coincides with $g_M$ in a neighborhood of $\partial M$, and satisfies $dV_{g_\eps}=dV_{g_M}$ in $M$ and $\sigma_1(M,g_{\ve}) \rightarrow + \infty$ as $\ve \rightarrow 0$.
\end{theorem}

It is noteworthy that the assumption on the dimension of the factors is essential. For instance, if $B$ is an interval and $F$ is a circle, then the product manifold $M=B \times F$ is a cylinder, and hence, it follows from Kokarev's inequality~\eqref{ineq:Kokarev} that $\sigma_1(M,g)L_g(\partial M)\leq 8\pi$.

  By establishing extensions of Theorem~\ref{thm:product} involving the mixed Steklov-Neumann spectrum, we are also able to study compact Riemannian manifolds with connected boundary isometric to a Riemannian product. In this setting, we show a result similar to Theorem \ref{thm:product}, where the Riemannian metrics have fixed volume instead of fixed volume element.
  
\begin{theorem}\label{thm:ProductBoundary}
Let $(M,g_M)$ be a compact Riemannian manifold with connected boundary $\partial M$  that is isometric to the product of two closed Riemannian manifolds of positive dimensions. Then for any $\ve > 0$ there exists a Riemannian metric $g_\ve$ on $M$ which coincides with $g_M$ on $\partial M$, satisfying $\Vol(M,g_\ve) = \Vol(M,g_M)$ and $\sigma_1(M,g_{\ve}) \rightarrow + \infty$ as $\ve \rightarrow 0$.
\end{theorem}

It is worth to point out that Theorem~\ref{thm:ProductBoundary} is applicable to a much larger class of manifolds than Theorem~\ref{thm:product}. Indeed, no asumption is made on the structure of the interior of the manifold and the only dimensional constraint is positivity. For instance, Theorem~\ref{thm:ProductBoundary} applies to any three-dimensional Riemannian manifold $M$ with boundary $\partial M$ that is isometric to a torus.

\subsection{The Steklov spectrum of warped products}
Let $B$ be a compact Riemannian manifold of dimension $n\geq 1$ with boundary $\partial B$ and let $F$ be a closed Riemannian manifold of dimension $k\geq 1$. Given a positive function $h \in C^\infty(B)$, the \emph{warped product} $M = B \times_h F$ is the product manifold $B\times F$ endowed with the Riemannian metric $g_B + h^2 g_F$. The Laplacian on a warped product is closely related to a diffusion operator $L_h$ on the base manifold $B$ determined by the warping function $h\in C^\infty(B)$. Indeed for $f\in C^\infty(B \times F)$, we have that
  $$\Delta_{M} f=L_hf + h^{-2} \Delta_F f,$$
  where $L_h=\Delta_B-\grad_B\ln h^k$. More precisely, this diffusion operator is given by
  $$L_hf=\Delta_Bf - g_B(\text{grad}_B \ln h^k,\text{grad}_Bf).$$
Some relations between the spectra of these operators in the setting of non-compact Riemannian manifolds have been established in \cite{Po2023,Po2021}. Motivated by them, we give an expression for the Steklov spectrum of a warped product in terms of this diffusion operator, which is a natural generalization of the standard description of the Steklov spectrum of a Riemannian product.

Denote by $\sigma_{\Delta}(F)=\{\lambda_j\}_{j\geq 0}$ the spectrum of the Laplacian on $F$.
  For each $\lambda\in\sigma_{\Delta}(F)$, consider the auxiliary Dirichlet-to-Neumann operator $\dtn_{h,\lambda}:C^\infty(\partial B)\to C^\infty(\partial B)$ defined by $\dtn_{h,\lambda}f=\partial_\nu\bar{f}$, where this time $\bar{f}$ is the solution of
   \begin{equation*}
     \begin{cases}
       L_h\bar{f} +\lambda h^{-2}\bar{f}= 0 &\text{ in } B, \\
       \bar{f} =  f &\text{ on } \partial B.\\
  \end{cases}
\end{equation*}
The spectrum of this alternative Dirichlet-to-Neumann map is also discrete and non-negative. If the dimension $n$ of the base $B$ is at least 2, then
$$\sigma_{\dtn_{h,\lambda}}(B)=\{\sigma_{h,\lambda,j}(B)\}_{j\geq 0}.$$
However, if the base is one-dimensional, then there are only two eigenvalues
$\sigma_{h,\lambda,0}(B)$ and $\sigma_{h,\lambda,1}(B)$, since the boundary $\partial B$ is 0-dimensional.

\begin{theorem}\label{thm:SteklovSpectrumWarpedProduct}
  Let $B$ be a compact Riemannian manifold with boundary and let $F$ be a closed Riemannian manifold of dimension $k\geq 1$. 
Then the Steklov spectrum of the warped product $M = B \times_h F$ is given by
  \begin{gather}\label{equation:warpedseparation}
    \sigma_{\dtn_\Delta}(M) = \bigcup_{\lambda \in \sigma_{\Delta}(F)} \sigma_{\dtn_{h,\lambda}}(B).
  \end{gather}

\end{theorem}

This result will follow from the more general Theorem~\ref{mixed Steklov-Neumann spectrum warped product}, which involves the mixed Steklov-Neumann spectrum.
For the purpose of this paper, the interest of this decomposition is that the dependence on a diffusion operator $L_h$ provides flexibility for certain constructions, even under quite restrictive geometric assumptions.
Moreover, working with warped metrics allowed us to obtain results in lower dimension than any previous methods. For instance, a key step in our work will be to prove a refined version of the following:
\begin{prop}
  Let $B$ be a compact Riemannian manifold with boundary. If the dimension $n$ of $B$ is at least 2,  there exists a family of smooth positive functions $(h_\eps)_{\eps>0}$ such that $h_\eps\equiv 1$ on $\partial B$ and $\sigma_1(\dtn_{L_{h_\eps}})\xrightarrow{\eps\to 0}+\infty$.
\end{prop}
See Proposition~\ref{diffusion large first eigenvalue}.
Let us stress that this proposition applies even for manifolds $B$ of dimension $n=2$. This is in contrast to conformal perturbations $g=\omega^2g_B$, where $\omega\in C^\infty(B)$ with $\omega_\eps=1$ on the boundary, for which Kokarev's bound~\eqref{ineq:Kokarev} prevent arbitrarily large spectral gap $\sigma_1$ if $n=2$.  This point is crucial in making Theorem~\ref{thm:product} and Theorem~\ref{thm:ProductBoundary} work for manifolds $M$ of dimension $m=3$.
See Remark~\ref{remark:conformal} for further discussion.

\subsection{Plan}
The paper is organized as follows: In Section \ref{Preliminaries}, we provide some preliminaries involving the mixed Steklov-Neumann spectrum of symmetric Laplace-type operators. In Section \ref{Warped products}, we focus on warped products and establish Theorem \ref{thm:SteklovSpectrumWarpedProduct}. Section \ref{Large} is devoted to Theorems \ref{thm:product} and \ref{thm:ProductBoundary}.

\subsection*{Acknowledgements}
The authors would like to thank Bruno Colbois for carefully reading a preliminary version of this paper. His comments lead to substantial improvements in the presentation.

\section{\bf Preliminary material}\label{Preliminaries}

Throughout this paper manifolds are assumed to be connected, with possibly non-connected, smooth boundary. The volume element corresponding to a Riemannian manifold $M$ will be denoted by $dV_M$, or $dV_{M,g}$ if more precision is needed regarding the metric, and sometimes also simply by $dV$ when the context is clear.

\subsection{Mixed Steklov--Neumann spectrum of diffusion and Laplace type operators}
Let $M$ be a compact Riemannian manifold with boundary $\Sigma$ and denote by $\nu$ the outward pointing unit normal to $\Sigma$. Write $\Sigma$ as the disjoint union of $\Sigma_{\St}$ and $\Sigma_{\Ne}$, where $\Sigma_{\St}$ and $\Sigma_{\Ne}$ are unions of connected components of $\Sigma$, with $\Sigma_{\St}$ non-empty and $\Sigma_{\Ne}$ possibly empty. A \emph{symmetric Laplace-type operator} on $M$ is an operator of the form
\[
L = \Delta - \grad \ln h + V,
\]
where $h \in C^{\infty}(M)$ is positive and $V \in C^{\infty}(M)$. Denote by $L^2_h(M)$ the space $L^2(M)$ endowed with the inner product
\[
\< u , v \>_{h} = \int_M uv\,hdV_M.
\]
It is immediate to verify that
\[
\< L u , v \>_h = \< u , L v \>_h = \int_M (\< d u , d v \> + Vuv)\,hdV_M
\]
for any $u,v \in C^{\infty}_{\Sigma_{\St}}(M)$, where $C^\infty_{\Sigma_{\St}}(M)$ stands for the space of smooth functions in $M$ vanishing on $\Sigma_{\St}$ and satisfying Neumann boundary condition on $\Sigma_{\Ne}$.

The operator
$L \colon C^{\infty}_{\Sigma_{\St}}(M) \subset L^2_h(M) \to L^2_h(M)$
admits a unique self-adjoint extension
\[
\bar{L} \colon D(\bar{L}) \subset L^2_h(M) \to L^2_h(M).
\]
The spectrum of this operator is discrete and consists of all $\lambda \in \R$ such that there exists a non-zero $f \in C^{\infty}_{\Sigma_{\St}}(M)$ satisfying $Lf = \lambda f$ in $M$. It is easy to see that if $V \geq 0$, then zero is not in the spectrum of this operator, which means that $\bar{L}$ is bijective.

In this case, it follows that any $f \in C^{\infty}(\Sigma_{\St})$ admits a unique extension $\bar{f} \in C^{\infty}(M)$ satisfying $L \bar{f} = 0$ in $M$ and $\partial_\nu \bar{f} = 0$ on $\Sigma_{\Ne}$. This gives rise to the \emph{modified Dirichlet-to-Neumann map}
\[
\mathcal{D}_L \colon C^{\infty}(\Sigma_{\St}) \subset L^2_h(\Sigma_{\St}) \to L^2_h(\Sigma_{\St})
\]
given by $\mathcal{D}_L f = \partial_\nu \bar{f}$. We readily see that
\[
\< \mathcal{D}_L  u , v \>_h = \< u , \mathcal{D}_L  v \>_h = \int_M (\< d \bar{u}, d \bar{v} \> + V \bar{u} \bar{v})\,hdV_M
\]
for any $u,v \in C^\infty(\Sigma_{\St})$. The operator $\mathcal{D}_L$ has a unique self-adjoint extension with discrete spectrum $\sigma(L;\Sigma_{\St},\Sigma_{\Ne})$. This is called the \emph{mixed Steklov-Neumann spectrum} of $L$ and consists of all $\sigma \in \R$ such that there exists a non-zero $f \in C^{\infty}(M)$ satisfying
\begin{equation*}
	\begin{cases}
		L f = 0 &\text{ in } M, \\
		\partial_\nu f =  \sigma f &\text{ on } \Sigma_{\St},\\
		\partial_\nu f = 0 &\text{ on } \Sigma_{\Ne}.
	\end{cases}
\end{equation*}
The eigenvalues of this operators are denoted by
$$0 \leq \sigma_0(L;\Sigma_{\St},\Sigma_{\Ne}) \leq \sigma_1(L;\Sigma_{\St}, \Sigma_{\Ne}) \leq \sigma_2(L;\Sigma_{\St},\Sigma_{\Ne}) \leq \dots,$$
repeated according to their multiplicities. In the case where $\Sigma_{\Ne}$ is empty, the spectrum of this operator is called the \emph{Steklov spectrum} of $L$. We denote this by $\sigma_{\dtn_L}(M)=\sigma(L;\Sigma_S,\emptyset)$ and by $0 \leq \sigma_0(L) \leq \sigma_1(L) \leq \sigma_2(L) \leq \dots$ the eigenvalues of the operator. In the case of the Laplacian $L = \Delta$, the spectrum of this operator is the \textit{mixed Steklov-Neumann spectrum of $M$} and we denote by $0 = \sigma_{0}(M;\Sigma_{\St},\Sigma_{\Ne}) \leq \sigma_{1}(M;\Sigma_{\St} , \Sigma_{\Ne}) \leq \dots$ the corresponding eigenvalues.

The following is an immediate consequence of Rayleigh's theorem, together with the fact that for any $f \in C^\infty(\Sigma_{\St})$, the extension $\bar{f} \in C^\infty(M)$ satisfying $L\bar{f} = 0$ in $M$ and $\partial_\nu \bar{f} = 0$ on $\Sigma_{\Ne}$ minimizes the Dirichlet energy corresponding to $L$ among all smooth extensions of $f$.

\begin{prop}\label{Laplace type first eigenvalue}
	Let $L = \Delta - \grad \ln h + V$ be a symmetric Laplace-type operator on a compact Riemannian manifold $M$ with boundary, where $V$ is non-negative.  Then the first mixed Steklov-Neumann eigenvalue of $L$ is given by
	\[
	\sigma_0(L;\Sigma_{\St},\Sigma_{\Ne}) = \min_{f} \frac{\int_M (\| d f \|^2 + V f^2)\,hdV_M}{\int_{\Sigma_{\St}} f^2\,hdV_{\Sigma}},
	\]
	where the minimum is taken over all non-zero $f \in C^\infty(M)$.
\end{prop}

In the case where $V = 0$, the operator $L = \Delta - \grad \ln h$ is called a \emph{diffusion operator} on $M$. We readily see that constant functions are eigenfunctions corresponding to $\sigma_0(L;\Sigma_{\St},\Sigma_{\Ne}) = 0$. Therefore, Rayleigh's theorem yields the following expression for the first non-zero Steklov eigenvalue.

\begin{prop}\label{diffusion first non-zero eigenvalue}
	Let $L = \Delta - \grad \ln h$ be a diffusion operator on a compact Riemannian manifold $M$ with boundary. Then the first non-zero mixed Steklov-Neumann eigenvalue of $L$ is given by
	\[
	\sigma_1(L;\Sigma_{\St},\Sigma_{\Ne}) = \min_{f} \frac{\int_{M} \| d f \|^2 hdV_M}{\int_{\Sigma_{\St}} f^2 hdV_\Sigma},
	\]
	where the minimum is taken over all non-zero $f \in C^\infty(M)$ which satisfy $\int_{\Sigma_{\St}} f hdV_\Sigma = 0$.
\end{prop}

\begin{rem}\label{remark:conformal}
Consider the conformal perturbation $g=\omega^2 g_M$, where $\omega\equiv 1$ on $\partial M$. Then, the Rayleigh--Steklov quotient associated to $\Delta_g$ is
$$\frac{\int_{M} \| d f \|_{g_M}^2 \omega^{m-2}dV_M}{\int_{\Sigma_{\St}} f^2 dV_\Sigma}.$$
Suppose that the dimension $m$ of $M$ is at least 3, and consider the diffusion operator $L= \Delta_{g_{M}} - \grad_{g_M} \ln h$ on $(M,g_M)$, where $h=\omega^{m-2}$. As a consequence of the min-max characterization of eigenvalues, we derive that the mixed Steklov-Neumann spectrum $\sigma(M,g;\Sigma_{\St},\Sigma_{\Ne})$ coincides with the mixed Steklov-Neumann spectrum $\sigma(L,g_M;\Sigma_{\St},\Sigma_{\Ne})$. This also follows from the fact that for any $f \in C^\infty(M)$ we have that $\Delta_{g} f = 0$ in $M$ if and only if $L f = 0$ in $M$.

One could then be tempted to conclude that we could have achieved the goals of this paper using conformal perturbations, without the introduction of diffusion operators and warped products, but this is misleading. Indeed, one of the strengths of our approach is that it also works when $M$ is two-dimensional.

\end{rem}

\subsection{Quasi-isometries and eigenvalues}
Two Riemannian metrics $g_1, g_2$ on $M$ are called \emph{quasi-isometric} with ratio $C \geq 1$ if
\[
\frac{1}{C} \leq \frac{g_1(X,X)}{g_2(X,X)} \leq C
\]
for any non-zero tangent vector $X$ of $M$. The following well-known proposition is a straightforward consequence of the min-max characterization of eigenvalues.

\begin{prop}\label{quasi-isometric}
Let $M^m$ be a compact manifold with boundary. If $g_1, g_2$ are quasi-isometric Riemannian metrics on $M$ with ratio $C\geq 1$, then the mixed Steklov-Neumann eigenvalues of $M$ with respect to these metrics are related by
	\[
	\frac{1}{C^{2m+1}} \leq \frac{\sigma_k(M,g_1;\Sigma_{\St} , \Sigma_{\Ne})}{\sigma_k(M,g_2;\Sigma_{\St} , \Sigma_{\Ne})} \leq C^{2m+1}
	\]
	for any $k \in \mathbb{N}$.
\end{prop}
In the context of the Steklov problem, this was previously used in~\cite{CoGiGi2019,CoGiRa2018} for instance.

\section{\bf The Steklov spectrum of warped products}\label{Warped products}
The purpose of this section is to prove a structure theorem for the mixed Steklov--Neumann problem of a warped product $B\times_hF$. Theorem~\ref{thm:SteklovSpectrumWarpedProduct} from the introduction will be obtained as a corollary.
Let $B$ be a compact Riemannian manifold with boundary $\Sigma$ and let $F$ a closed Riemannian manifold of dimension $k$. Given a positive function $h \in C^\infty(B)$, the \emph{warped product} $M = B \times_h F$ is the product manifold $B\times F$ endowed with the Riemannian metric $g_B + h^2 g_F$.
Elementary geometric properties of warped products may be found for example in \cite[Section 1.4]{FaIaPa2004}.
Note that a part of this section also applies to the total space of a Riemannian submersion with fibers of basic mean curvature. In that case, the diffusion operators will involve the mean curvature of the fibers rather than the warping function.

Let $p \colon M \to B$ be the projection and denote by $F_x = p^{-1}(x)$ the fiber over $x \in B$. Given $f \in C^\infty(B)$, set $\tilde{f} = f \circ p$. Moreover, for a vector field $X$ on $B$, we denote by $\tilde{X}$ the corresponding vector field on $M$, under the usual identification of the tangent space of $M$ with the sum of the tangent spaces of $B$ and $F$. 

The unnormalized mean curvature vector field of $F_x$ at $(x,y) \in M$ is given by
\[
H = \sum_{i=1}^k \alpha(e_i,e_i) = - k \grad \ln \tilde{h},
\]
where $\alpha$ is the second fundamental form of the fiber $F_x\subset M$ and $\{e_i\}_{i=1}^k$ is an orthonormal basis of $T_{(x,y)}F_x$.
It is straightforward to compute (cf. for instance, \cite[Subsection 2.2]{BeMoPi2012}) that the gradient and the Laplacian of any $f \in C^\infty(B)$ and its lift $\tilde{f} \in C^\infty(M)$ are related by
\begin{equation*}
	\grad \tilde{f} = \widetilde{\grad f} \text{ and } \Delta \tilde{f} = \widetilde{\Delta f} + \< H , \grad \tilde{f} \>.
\end{equation*}
Considering the diffusion operator $L = \Delta - \grad \ln h^k$ on $B$, the latter one may be rewritten as
\begin{equation}\label{Laplacian of lift}
	\Delta \tilde{f} = \widetilde{Lf}
\end{equation}
for any $f \in C^\infty(B)$.
The next result is a generalisation of Theorem~\ref{thm:SteklovSpectrumWarpedProduct} from the introduction.
\begin{theorem}\label{mixed Steklov-Neumann spectrum warped product}
 Let $B$ be a compact Riemannian manifold with boundary $\Sigma$ written as the disjoint union of\, $\Sigma_{\St}$ and $\Sigma_{\Ne}$, where $\Sigma_{\St}$ and $\Sigma_{\Ne}$ are unions of connected components of\, $\Sigma$, with $\Sigma_{\St}$ non-empty and $\Sigma_{\Ne}$ possibly empty. Let $F^k$ a closed Riemannian manifold. 
  Given a positive function $h \in C^\infty(B)$, the mixed Steklov-Neumann spectrum of the warped product $M = B \times_h F$ is given by
	\[
	\sigma(M;\Sigma_{\St} \times F, \Sigma_{\Ne} \times F) = \bigcup_{\lambda \in \sigma_{\Delta}(F)} \sigma(L +\lambda h^{-2} ; \Sigma_{\St},\Sigma_{\Ne}),
	\]
	where $L = \Delta - \grad \ln h^k$ on $B$.
\end{theorem}

\begin{proof}
Let $0 = \lambda_0(F) < \lambda_1(F) \leq \dots$ be the eigenvalues of the Laplacian on $F$ and consider an orthonormal basis $\{\vf_i\}_{i=0}^{\infty}$ of $L^2(F)$ consisting of eigenfunctions of the Laplacian on $F$. Denote by $q \colon M \to F$ the projection and set $\bar{\vf}_i = \vf_i \circ q$. At a point $(x,y) \in M$, keeping in mind that $B\times \{y\}$ is a totally geodesic submanifold of $M$, we compute
	\begin{equation}\label{Laplacian of lift 2}
		\Delta \bar{\vf}_i = \Delta \vf_i|_{F_x} (y) = h^{-2}(x) \Delta_F \vf_i (y) = \lambda_i(F)\tilde{h}^{-2} \bar{\vf},
	\end{equation}
	where we used that the Riemannian metric of $F_x\subset M$ is the Riemannian metric of $F$ multiplied by the constant $h^2(x)$.

        For any $i \geq 0$, denote by $\sigma_{i,0} \leq \sigma_{i,1} \leq \dots$ the mixed Steklov-Neumann eigenvalues of $L + \lambda_i(F)h^{-2}$ on $B$ and let $\{ f_{i,j} \}_{j = 0}^{\infty}$  be an orthonormal basis of $L^2_{h^k}(\Sigma_{\St})$ consisting of eigenfunctions of the corresponding Dirichlet-to-Neumann operator $\dtn_L$. Extend $f_{i,j}$ to $\hat{f}_{i,j} \in C^\infty(B)$ satisfying
	\begin{equation*}
          \begin{cases}
            (L + \lambda_i(F)h^{-2})\tilde{f}_{i,j} = 0 &\text{ in } B, \\
            \partial_\nu\tilde{f}_{i,j} =  \sigma_{i,j} f_{i,j} &\text{ on } \Sigma_{\St},\\
            \partial_\nu\tilde{f}_{i,j} = 0 & \text{ on } \Sigma_{\Ne}.
          \end{cases}
	\end{equation*}
	
	Setting $\tilde{f}_{i,j} = \hat{f}_{i,j} \circ p$, it follows from (\ref{Laplacian of lift}) and (\ref{Laplacian of lift 2}) that
	\[
	\Delta (\bar{\vf}_i \tilde{f}_{i,j}) = \bar{\vf}_i \Delta \tilde{f}_{i,j} + \tilde{f}_{i,j} \Delta \bar{\vf}_i = \bar{\vf} \widetilde{L\hat{f}}_{i,j} + \lambda_i(F) \tilde{h}^{-2} \tilde{f}_{i,j} \bar{\vf} = 0
	\]
in $M$, while
\begin{equation*}
	\partial_{\nu} (\bar{\vf}_i \tilde{f}_{i,j}) = \bar{\vf}_i \partial_{\nu}  \tilde{f}_{i,j} = \begin{cases}
	\sigma_{i,j} \bar{\vf}_i \tilde{f}_{i,j} &\text{ on } \Sigma_{\St} \times F, \\
	0 &\text{ on } \Sigma_{\Ne} \times F.
	\end{cases}
\end{equation*}
For all $i,j \geq 0$, this means that $\bar{\vf}_i \tilde{f}_{i,j}$ is a mixed Steklov-Neumann eigenfunction of $M$ corresponding to the eigenvalue $\sigma_{i,j}$.

Keeping in mind that
\[
\int_{\Sigma_{\St}\times_h F} u^2\,dV_{\Sigma_{\St}\times_h F} \leq \max_{\partial B} h^k \int_{\Sigma_{\St} \times F} u^2\,dV_{\Sigma_{\St}\times F},
\]
for any $u \in L^2(\Sigma_{\St}\times_h F)$, we readily see that the space spanned by functions of the form $\tilde{\vf} \tilde{f} \in C^\infty(\Sigma_{\St}\times_h F)$ with $\vf \in C^\infty(F)$ and $f \in C^\infty(\Sigma_{\St})$ is dense in $L^2(\Sigma_{\St}\times_h F)$, being dense in $L^2(\Sigma_{\St} \times F)$. Consider now a non-zero function $\vf \in C^\infty(F)$, $f \in C^\infty(\Sigma_{\St})$ and $\ve > 0$. Then there exists $n_0 \in \mathbb{N}$ and $a_0,\dots,a_{n_0} \in \mathbb{R}$ such that
\[
\big\| \vf - \sum_{i=0}^{n_{0}} a_i \vf_i \big\|_{L^2(F)} < \frac{\ve}{2\| f \|_{L^2_{h^k}(\Sigma_{\St})}},
\]
$\{ \vf_i \}_{i=0}^\infty$ being an orthonormal basis of $L^2(F)$. Similarly, for any $i=0,\dots,n_0$ there exist $m(i) \in \mathbb{N}$ and $b_{i,0},\dots,b_{i,m(i)}\in \mathbb{R}$ such that
\[
\big\| f - \sum_{j=0}^{m(i)} b_j f_{i,j} \big\|_{L_{h^k}^2(\Sigma_{\St})} < \frac{\ve}{2(n_0 + 1)\| \vf \|_{L^2(F)}}.
\]
Using that $\| \tilde{\vf} \tilde{f} \|_{L^2(\Sigma_{\St} \times_h F)} = \| \vf \|_{L^2(F)} \| f \|_{L^2_{h^k}(\Sigma_{\St})}$, this gives the estimate
\begin{eqnarray}
\big\| \tilde{\vf} \tilde{f} - \sum_{i=0}^{n_0} \sum_{j=0}^{m(i)} a_i b_{i,j} \tilde{\vf}_i \tilde{f}_{i,j} \big\|_{L^2(\Sigma_{\St} \times_h F)} &\leq& \big\| \big( \tilde{\vf} - \sum_{i=0}^{n_0} a_i \tilde{\vf} \big) \tilde{f} \big\|_{L^2(\Sigma_{\St} \times_h F)} \nonumber\\
&+& \big\| \sum_{i=1}^{n_0} a_i \tilde{\vf}_i \big( f - \sum_{j=0}^{m(i)} b_{i,j} \tilde{f}_{i,j} \big)  \big\|_{L^2(\Sigma_{\St} \times_h F)} \nonumber\\
&<& \ve, \nonumber
\end{eqnarray}
which yields that $\{\bar{\vf}_i \tilde{f}_{i,j}\}_{i,j=0}^\infty$ spans a dense subspace of $L^2(\Sigma_{\St} \times_h F)$. Finally, it is immediate to verify that
	\[
	\int_{\Sigma_{\St} \times_h F} \bar{\vf}_i \tilde{f}_{i,j}  \bar{\vf}_m \tilde{f}_{m,l} = \int_{\Sigma_{\St}} f_{i,j} f_{m,l} h^k \int_{F} \vf_i \vf_m = \delta_{im} \delta_{jl},
	\]
which means that $\{\bar{\vf}_i \tilde{f}_{i,j}\}_{i,j=0}^\infty$ is an orthonormal basis of $L^2(\Sigma_{\St} \times_h F)$ consisting of mixed Steklov-Neumann eigenfunctions.
\end{proof}

\section{\bf Large eigenvalues on product manifolds}\label{Large}

The aim of this section is to prove Theorem~\ref{thm:product} and Theorem~\ref{thm:ProductBoundary}. In order to unify our approach, we will establish an extension of Theorem \ref{thm:product} involving the mixed Steklov-Neumann spectrum. Although it is not assumed in Theorem \ref{thm:product}, throughout most of this section, we focus on the case where the boundary $\Sigma$ of $B^n$ has a neighborhood isometric to $\Sigma \times [0,\ell)$ for some $\ell > 0$. By virtue of Proposition \ref{quasi-isometric}, this is not restrictive for our purposes.

We begin by establishing the existence of diffusion operators with arbitrarily large first non-zero mixed Steklov-Neumann eigenvalue on such a Riemannian manifold $B$, if $n \geq 2$. This is conceptually related to \cite[Theorem 1.1]{CoElGi2019}, where conformal Riemannian metrics are considered, and the proof is quite similar. See Remark~\ref{remark:conformal} for a relevant discussion. For reasons that will become clear in the sequel, we are also interested in certain properties of the function $h$ defining the diffusion operator, which are not discussed in \cite[Theorem 1.1]{CoElGi2019}. For this reason, we present a detailed proof to the following:

\begin{prop}\label{diffusion large first eigenvalue}
Let $B^n$ be a Riemannian manifold of dimension $n\geq 2$ with boundary $\Sigma = \Sigma_{\St} \cup \Sigma_{\Ne}$ as above such that there exists a neighborhood $U$ of\, $\Sigma$ isometric to $\Sigma \times [0,\ell)$. Then for any $0 <\ve < \frac{\ell}{6}$ and any $0 < \delta < 1$ there exists a family of functions $h=h_{\eps,\delta} \in C^{\infty}(B)$ depending only on the distance to $\Sigma$, with $h = 1$ in a neighborhood of\, $\Sigma$, $h \geq \ve^\delta$ in $B$ and $h = \ve^\delta $ in $\Sigma \times [\ve , 2 \ve]$, such that the first non-zero mixed Steklov-Neumann eigenvalue of the operator $L = \Delta - \grad \ln h$ satisfies $\sigma_1(L;\Sigma_{\St} , \Sigma_{\Ne}) \rightarrow + \infty$ as $\ve \rightarrow 0$.
\end{prop}
\begin{rem}
  The control of the function $h_{\eps,\delta}$ will be used in the proof of Proposition~\ref{aux result}, where the constant $\delta$ will be chosen so that $k/n<\delta<1$ to insure that a natural dichotomy leads to arbitrarily large lower bound.
\end{rem}
\begin{proof}
Without loss of generality, we assume $\ell < 1$.
  Choose $h \in C^{\infty}(B)$ depending only on the distance to $\Sigma$, with $h = 1$ in $\Sigma \times [0 ,\ve/2]$, $h \geq \ve^\delta$ in $B$, $h = \ve^\delta$ in $\Sigma \times [\ve , 2 \ve]$, and $h = \ve^{-2}$ in $B \smallsetminus (\Sigma \times [0,3\ve))$. In view of Proposition~\ref{diffusion first non-zero eigenvalue}, it suffices to show that there exists $C(\ve) > 0$ such that $C(\ve) \rightarrow + \infty$ as $\ve \rightarrow 0$, satisfying
\begin{equation}\label{energy estimate}
\mathcal{R}(f):=\int_B \| d f \|^2\,hdV_B \geq C(\ve)
\end{equation}
for any $f \in C^\infty(B)$ with $\int_{\Sigma_{\St}} f^2\,dV_\Sigma = 1$ and $\int_{\Sigma_{\St}} f\,dV_\Sigma = 0$. Denote by $\Sigma_j$ the connected components of $\Sigma_{\St}$, $0 \leq j \leq b$, and let
$$0 = \lambda_0(\Sigma_{\St})=\cdots=\lambda_b(\Sigma_{\St})<\lambda_{b+1}(\Sigma_{\St}) \leq \lambda_{b+2}(\Sigma_{\St})\leq\dots$$
be the eigenvalues of the Laplacian on $\Sigma_{\St}$ and let $\{\vf_j\}_{j=0}^{+\infty}$ be an orthonormal basis of $L^2(\Sigma_{\St})$ consisting of eigenfunctions, where
\begin{equation*}
	{\vf}_j = \begin{cases}
		|\Sigma_j|^{-1/2} & \text{ in } \Sigma_j, \\
		0 &\text{ in } \Sigma_{\St} \smallsetminus \Sigma_j,
	\end{cases}
\end{equation*}
for $0\leq j \leq b$. We write points of $U$ in the form $(x,t)$ with $x \in \Sigma$ and $0 \leq t < \ell$. Then a test function $f$ (as above) restricted to $\Sigma_{\St} \times [0,\ell)$ is developed in Fourier series as
\[
f(x,t) = \sum_{j \geq 0} a_j(t) {\vf}_j(x), \text{ where } a_j(t) = \int_{\Sigma_{\St} \times \{t\}} f {\vf}_j\,dV_\Sigma.
\]
The conditions $\int_{\Sigma_{\St}} f^2\,dV_\Sigma = 1$ and $\int_{\Sigma_{\St}} f\,dV_\Sigma = 0$ satisfied by $f$ may be rewritten as
\[
\sum_{j \geq 0} a_j^2(0) = 1 \text{ and } \sum_{j \leq b} a_j(0)|\Sigma_j|^{1/2} = 0.
\]
In particular, the situation when the boundary $\Sigma$ is connected corresponds to $b=0$, with $a_0(0)=0$.
The fact that
\[
df(x,t) = \sum_{j \geq 0} \bigl(a_j^\prime(t) {\vf}_j(x) \dt + a_j(t) d_{\Sigma_{\St}} {\vf}_j(x)\bigr)
\]
in $U$, gives that
\begin{gather}\label{ineq:RayleighProof}
  \mathcal{R}(f) \geq \sum_{j \geq 0} \int_{0}^\ell \bigl( a_j^\prime(t)^2 + a_j^2(t) \lambda_j( \Sigma_{\St}) \bigr) h(t) \dt,
\end{gather}
where we used that $h$ depends only on the distance to $\Sigma$, and by abuse of notation write $h(x,t) = h(t)$ for any $(x,t) \in U$.

Fix $j \geq 0$. If there exists $t_0 \in (3 \ve,4 \ve)$ such that $| a_j(t_0) | \leq |a_j(0)|/2$, then
\begin{gather}\label{ineq:proofterm1}
  \int_0^\ell a_j^\prime(t)^2 h(t) \dt \geq \ve^{\delta} \int_{0}^{t_0} a_j^\prime(t)^2 \dt \geq \frac{\ve^{\delta}}{t_0} \left( \int_0^{t_0} a_j^\prime(t) \dt \right)^2 \geq \frac{\ve^{\delta - 1}}{16} a_j^{2}(0),
\end{gather}
where we used the Cauchy--Schwarz inequality and that $h \geq \ve^\delta$ in $B$. Otherwise, that is, if $|a_j(t)| \geq |a_j(0)|/2$ for any $t \in (3 \ve,4\ve)$, we obtain from $h = \ve^{-2}$ in $\Sigma \times [3 \ve,4 \ve]$ that

\begin{gather}\label{ineq:proofterm2}
  \int_{0}^\ell a_j^2(t) \lambda_j(\Sigma_{\St}) h(t) \dt \geq \frac{a_j^2(0)\lambda_j( \Sigma_{\St})}{4} \int_{3 \ve}^{4 \ve} h(t) \dt = \frac{\lambda_j(\Sigma_{\St})}{4 \ve} a_j^{2}(0) 
\end{gather}
In both cases, it follows from \eqref{ineq:proofterm1} and \eqref{ineq:proofterm2} via \eqref{ineq:RayleighProof} that
\[
\int_{0}^\ell ( a_j^\prime(t)^2 + a_j^2(t) \lambda_j( \Sigma_{\St}) ) h(t) \dt \geq C_1 \ve^{\delta - 1} a_j^{2}(0)
\]
for any $j > b$, where
\[
 C_1 = \min\left\{ \frac{1}{16} , \frac{\lambda_{b+1}(\Sigma_{\St})}{4} \right\} .
\]
This yields that
\[
\mathcal{R}(f) \geq C_1 \ve^{\delta - 1} \sum_{j > b} a_j^2(0) =  C_1 \ve^{\delta - 1} \left( 1 - \sum_{j \leq b} a_j^2(0)\right),
\]
and thus, $\mathcal{R}(f) \geq C_1 \ve^{\delta - 1}/2$ if $\sum_{j \leq b} a_j^2(0) \leq 1/2$. It is worth to point out that this holds if $\Sigma_{\St}$ is connected, that is, if $b=0$.

So it remains to deal with the case where  $\sum_{j \leq b} a_j^2(0) > 1/2$. This implies that there exists some $0 \leq i \leq b$ such that $a_i^2(0) \geq 1/(2(b+1))$. Hence, without loss of generality, we may suppose that $a_0(0) \geq 1/\sqrt{2(b+1)} > 0$. Similarly, since
\[
\sum_{i=1}^b a_j(0) |\Sigma_j|^{1/2} = - a_0(0) |\Sigma_0|^{1/2},
\]
we may assume that
\[
 a_1(0) |\Sigma_1|^{1/2} \leq - \frac{|\Sigma_0|^{1/2}}{b \sqrt{2(b+1)}} < 0.
\]

If there exists $t_0 \in (\ell/2,\ell) \subset (3\ve,\ell)$ such that $a_0(t_0) \leq a_0(0)/2$, keeping in mind that $h \geq \ve^\delta$ in $B$ and $h = \ve^{-2}$ in $B \smallsetminus (\Sigma \times [0,3\ve))$, we calculate
\begin{eqnarray}
	\mathcal{R}(f) &\geq& \int_{0}^{t_0} a_0^\prime(t)^2 h(t)\dt \geq \ve^\delta \int_0^{3\ve} a_0^\prime(t)^2 \dt + \ve^{-2} \int_{3\ve}^{t_0} a_0^\prime(t)^2\dt \nonumber\\
	&\geq& \frac{\ve^{\delta - 1}}{3} \left(\int_0^{3\ve}a_0^\prime(t) \dt \right)^2 + \frac{\ve^{-2}}{\ell} \left(\int_{3\ve}^{t_0}a_0^\prime(t) \dt \right)^2 \nonumber\\
	& \geq & \frac{\ve^{\delta - 1}}{3} \big((a_0(3\ve) - a_0(0))^2 + (a_0(t_0) - a_0(3\ve))^2\big) \nonumber \\
	&\geq& \frac{\ve^{\delta-1}}{6} (a_0(t_0) - a_0(0))^2 \geq \frac{\ve^{\delta - 1}}{48(b+1)}, \nonumber
\end{eqnarray}
where we used that $\ve <  1$ and $\ell < 1$.
Arguing in a similar way, if there exists $t_0 \in (\ell/2,\ell) \subset (3\ve,\ell)$ such that $a_1(t_0) \geq a_1(0)/2$, we derive that
\[
\mathcal{R}(f) \geq \int_{0}^{t_0} a_1^\prime(t)^2 h(t) \dt \geq \frac{|\Sigma_0| \ve^{\delta - 1}}{48 b^2 (b+1)|\Sigma_1|}.
\]
In the aforementioned cases, it is evident that $\mathcal{R}(f) \geq C_2 \ve^{\delta - 1}$, where
\[
C_2 = \min\left\{ \frac{1}{48(b+1)} ,  \frac{1}{48 b^2 (b+1)} \min_{0 \leq i,j \leq b} \frac{|\Sigma_i|}{|\Sigma_j|} \right\}.
\]

In the remaining case, that is, if $a_0(t) > a_0(0) / 2$ and $a_1(t) < a_1(0)/2$ for any $t \in (\ell/2,  \ell)$, the mean value of $f$ satisfies
		\[
\fint_{\Sigma_0 \times (\ell/2, \ell)} f\,dV= \frac{|\Sigma_0|^{1/2}}{| \Sigma_0 \times (\ell/2 , \ell)  |} \int_{\ell/2}^{\ell} a_0(t) \dt \geq \frac{ a_0(0)}{2|\Sigma_0|^{1/2}} > 0
\]
and
\[
\fint_{\Sigma_1 \times (\ell/2 , \ell)} f\,dV= \frac{|\Sigma_1|^{1/2}}{| \Sigma_1 \times (\ell/2 , \ell)  |} \int_{\ell/2}^{\ell} a_1(t) \dt \leq \frac{a_1(0)}{2|\Sigma_1|^{1/2}}  < 0.
\]
We derive from \cite[Lemma 3.4]{CoElGi2019} applied to $f$ on $M = B \smallsetminus \Sigma \times [0,\ell/2)$ that
\[
\int_{M} \| df \|^2\,dV_B \geq \frac{\ell \mu(M) }{16} \min\{ |\Sigma_0| ,|\Sigma_1| \} \left( \frac{a_0(0)}{|\Sigma_0|^{1/2}}  - \frac{ a_1(0)}{|\Sigma_1|^{1/2}}  \right)^{2},
\]
where $\mu(M)$ stands for the first non-zero Neumann eigenvalue of $M$. Since 
\[
\frac{a_0(0)}{|\Sigma_0|^{1/2}}  - \frac{a_1(0)}{|\Sigma_1|^{1/2}}  \geq  \frac{a_0(0)}{|\Sigma_0|^{1/2}} \geq \frac{1}{|\Sigma_0|^{1/2} \sqrt{2(b+1)}},
\]
and $h = \ve^{-2}$ in $M$
we deduce that 
\[
\mathcal{R}(f) \geq \ve^{-2} \int_{M} \| df \|^2\,dV_B \geq C_3 \ve^{-2} \geq C_{3} \ve^{\delta - 1},
\]
where
\[
C_3 = \frac{\ell \mu(M) }{32 (b+1)} \min_{0 \leq i,j\leq b} \frac{|\Sigma_i|}{|\Sigma_j|}.
\]
We conclude that (\ref{energy estimate}) holds for $C(\ve) = \min\{C_1/2,C_2,C_3\} \ve^{\delta -1}$, which completes the proof.
\end{proof}

\begin{prop}\label{aux result}
  Let $M$ be the Riemannian product of a compact manifold $B^n$ with boundary $\Sigma = \Sigma_{\St} \cup \Sigma_{\Ne}$ and let $F^k$ be a closed manifold, where $n > k\geq 1$. Suppose that $\Sigma$ has a neighborhood $U\subset B$ isometric to $\Sigma \times [0,\ell)$. Then for any $\ve > 0$ there exists a positive function $h_\eps \in C^{\infty}(B)$ with $h = 1$ in a neighborhood of\, $\Sigma$, such that the first non-zero mixed Steklov-Neumann eigenvalue of $M$ with respect to the Riemannian metric
  $$g_{\ve} = h_\eps^{-2k/n} g_B + h_\eps^2 g_{F}$$
  satisfies
  $\sigma_1(M,g_{\ve};\Sigma_{\St} \times F , \Sigma_{\Ne} \times F) \rightarrow + \infty$ as $\ve \rightarrow 0$, while the volume element $dV_{M,g_\eps}$ does not depend on $\eps$.
\end{prop}

\begin{proof} 
	Let $h \in C^{\infty}(B)$ be a positive function with $h=1$ on $\Sigma$. Then $M$ endowed with the Riemannian metric $g^\prime = h^{-2k/n} g_B + h^2 g_{F}$ is isometric to the warped product $(B,g_B^\prime) \times_h F$, where $g_B^\prime = h^{-2k/n}g_B$. Considering the diffusion operator $L^\prime = \Delta_{g^\prime_B} - \grad_{g^\prime_B} \ln h^k$ on $(B,g_B^\prime)$, Theorem \ref{thm:SteklovSpectrumWarpedProduct} asserts that the mixed Steklov-Neumann spectrum of $(M,g^\prime)$ is the union of $\sigma(L^\prime + \lambda h^{-2};\Sigma_{\St},\Sigma_{\Ne})$ with $\lambda \in \sigma_{\Delta}(F)$. It follows from Proposition \ref{Laplace type first eigenvalue} that $\sigma_0(L^\prime + \lambda h^{-2};\Sigma_{\St},\Sigma_{\Ne})$ is non-decreasing with respect to $\lambda$, which yields that
	\[
	\sigma_1(M, g^\prime;\Sigma_{\St} \times F , \Sigma_{\Ne} \times F ) = \min\{ \sigma_1(L^\prime;\Sigma_{\St},\Sigma_{\Ne}) , \sigma_{0}(L^\prime + \lambda_1(F) h^{-2} ; \Sigma_{\St} , \Sigma_{\Ne}) \},
	\]
	where $\lambda_1(F)$ is the first non-zero eigenvalue of the Laplacian on $F$. Proposition \ref{diffusion first non-zero eigenvalue} states that
          \[
            \sigma_1(L^\prime;\Sigma_{\St},\Sigma_{\Ne}) = \min_{f} \int_{B} \| df \|_{g_B^\prime}^2 h^k\,dV_{B,g_B^\prime} = \min_{f} \int_{B} \| df \|_{g_B}^2 h^{2k/n}\,dV_B,
          \]
          where the minimum is taken over all $f \in C^{\infty}(B)$ satisfying $\int_{\Sigma_{\St}} f^2\,dV_\Sigma = 1$ and $\int_{\Sigma_{\St}} f\,dV_\Sigma = 0$.
        Considering the diffusion operator $L = \Delta - \grad \ln h^{2k/n}$ on $B$ endowed with $g_B$, it follows from Proposition \ref{diffusion first non-zero eigenvalue} that
        $$\sigma_1(L^\prime ; \Sigma_{\St} , \Sigma_{\Ne}) = \sigma_1(L;\Sigma_{\St},\Sigma_{\Ne}).$$
        See also Remark~\ref{remark:conformal}.
        Fix $ k/n < \delta < 1$. We know from Proposition \ref{diffusion large first eigenvalue} that any $0 <\ve < \ell/6$ there exists $h \in C^{\infty}(B)$ depending only on the distance to $\Sigma$, with $h = 1$ in a neighborhood of $\Sigma$, $h^{2k/n} \geq \ve^\delta$ in $B$ and $h^{2k/n} = \ve^\delta $ in $\Sigma \times [\ve , 2 \ve]$, such that $\sigma_1(L^\prime;\Sigma_{\St} , \Sigma_{\Ne}) = \sigma_1(L;\Sigma_{\St},\Sigma_{\Ne}) \rightarrow + \infty$ as $\ve \rightarrow 0$.
	
Hence, it remains to prove that for such choice of $h \in C^\infty(N)$ we have that $\sigma_0(L^\prime + \lambda_1(F) h^{-2} ; \Sigma_{\St},\Sigma_{\Ne}) \rightarrow + \infty$ as $\ve \rightarrow 0$. We obtain from Proposition \ref{Laplace type first eigenvalue} that
\begin{eqnarray}
  \sigma_0(L^\prime + \lambda_1(F) h^{-2} ; \Sigma_{\St},\Sigma_{\Ne}) &=& \min_{f} \int_{B} (\| df \|_{g^\prime}^2 h^k + \lambda_1(F) f^2  h^{k-2})\,dV_{B,g^\prime}\nonumber \\
                                                                       &=& \min_{f} \int_{B} (\| df \|_{g_B}^2 h^{2k/n} + \lambda_1(F)  f^2 h^{-2})\,dV_B, \nonumber
\end{eqnarray}
where the minimum is taken over all $f \in C^\infty(B)$ with $\int_{\Sigma_{\St}} f^2\,dV_\Sigma = 1$. Consider $f \in C^{\infty}(B)$ with $\int_{\Sigma_{\St}} f^2\,dV_\Sigma = 1$, and set
	\[
	\mathcal{R}(f) := \int_{B} (\| df \|_{g_B}^2 h^{2k/n} + \lambda_1(F)  f^2 h^{-2})\,dV_B.
	\]
	Denote by $0=\lambda_0(\Sigma) \leq \lambda_1(\Sigma) \leq \dots$ the eigenvalues of the Laplacian on $\Sigma_{\St}$ and let $\{\vf_j\}_{j=0}^{+\infty}$ be an orthonormal basis of $L^2(\Sigma_{\St})$ consisting of eigenfunctions. Writing points of $U$ in the form $(x,t)$ with $x \in \Sigma$ and $0 \leq t < \ell$, the restriction of $f$ to $\Sigma_{\St} \times [0,\ell)$ is developed in Fourier series as
	\[
	f(x,t) = \sum_{j \geq 0} a_j(t) {\vf}_j(x), \text{ where } a_j(t) = \int_{\Sigma_{\St} \times \{t\}} f {\vf}_j\,dV_\Sigma.
	\]
Since $\int_{\Sigma_{\St}} f^2\,dV_\Sigma = 1$, we readily see that $\sum_{j \geq 0} a_j(0)^2 = 1$. From the fact that
	\[
	df(x,t) = \sum_{j \geq 0} (a_j^\prime(t) {\vf}_j(x) \dt + a_j(t) d_{\Sigma}{\vf}_j(x))
	\]
	in $\Sigma_{\St} \times [0,\ell)$, writing $h(x,t)=h(t)$ for $(x,t) \in U$, we derive that
        \begin{eqnarray}
          \mathcal{R}(f) &\geq& \int_{U} (\| df \|^2 h^{2k/n} + \lambda_1(F)  f^2 h^{-2})\,dV_B \nonumber\\
	    &=& \sum_{j \geq 0} \int_0^\ell ((a_j^\prime(t)^2  + \lambda_j(\Sigma) a_j(t)^2) h(t)^{2k/n} +  \lambda_1(F)  a_j(t)^2 h(t)^{-2}) \dt \nonumber \\
	    &\geq& \sum_{j \geq 0} \int_0^\ell (a_j^\prime(t)^2 h(t)^{2k/n}   + \lambda_1(F) a_j(t)^2 h(t)^{-2} ) \dt. \nonumber
	  \end{eqnarray}
          Fix $j \geq 0$. If there exists $t_0 \in (\ve , 2 \ve )$ such that $| a_j(t_0) | < |a_j(0)|/2$, then
	\[
	\int_{0}^{\ell}  a_j^\prime(t)^2 h(t)^{2k/n} \dt \geq \ve^\delta  \int_{0}^{t_0} a_j^\prime(t)^2\dt \geq \frac{\ve^\delta}{t_0} \bigg(\int_{0}^{t_0} a_j^\prime(t) \dt \bigg)^2 \geq \frac{\ve^{\delta - 1}}{8} a_j(0)^2.
	\]
	Otherwise, we have that $|a_j(t)| \geq |a_j(0)|/2$ for all $t \in (\ve , 2 \ve)$, which yields that
	\[
	\int_0^{\ell} \lambda_1(F)  a_j(t)^2 h(t)^{-2} \dt \geq \frac{a_j(0)^2}{4} \lambda_1(F) \int_{\ve}^{2\ve} h(t)^{-2} \dt =\frac{\lambda_1(F) \ve^{1 - \delta n/k}}{4}a_j(0)^2,
	\]
	where we used that $h^{2k/n} = \ve^{\delta}$ in $\Sigma \times [\ve , 2 \ve]$. In any case, we obtain that
	\[
	\int_0^\ell (a_j^\prime(t)^2 h(t)^{2k/n}   + \lambda_1(F) a_j(t)^2 h(t)^{-2} ) \dt \geq C(\ve) a_j^2(0)
	\]
	for any $j \geq 0$, where
	\[
	C(\ve) = \min\left\{ \frac{\ve^{\delta - 1}}{8} ,  \frac{\lambda_1(F) \ve^{1 - \delta n /k}}{4}  \right\}.
	\]
This, together with $\sum_{j \geq 0} a_j(0)^2 = 1$, gives that $\mathcal{R}(f) \geq C(\ve)$, and thus,
	\[
	\sigma_0(L^\prime + \lambda_1(F)h^{-2};\Sigma_{\St},\Sigma_{\Ne}) \geq C(\ve) \rightarrow + \infty
	\]
as $\ve \rightarrow 0$, from the choice of $\delta$.
\end{proof}

We are ready to establish the following generalization of Theorem \ref{thm:product}.

\begin{theorem}\label{thm:product mixed}
Let $(M,g)$ be the Riemannian product of a compact manifold $B^n$ with boundary $\Sigma = \Sigma_{\St} \cup \Sigma_{\Ne}$ and let $F^k$ be a closed manifold, where $n > k\geq 1$. Then for any $\ve > 0$ there exists a Riemannian metric $g_\ve$ on $M$ having the same volume element as $g$, which coincides with $g$ in a neighborhood of\, $\Sigma$, such that $\sigma_1(M,g_{\ve};\Sigma_{\St} \times F , \Sigma_{\Ne} \times F) \rightarrow + \infty$ as $\ve \rightarrow 0$.
\end{theorem}

\begin{proof}
Since $B$ is compact, its Riemannian metric $g_B$ is quasi-isometric with ratio $C \geq 1$ to a Riemannian metric $g_B^\prime$, such that with respect to $g_B^\prime$, $\Sigma$ has a neighborhood isometric to $\Sigma \times [0,\ell)$. Proposition \ref{aux result} asserts that for any $\ve > 0$ there exists a positive $h=h_\eps \in C^{\infty}(B)$ with $h=1$ in a neighborhood of $\Sigma$, such that the first non-zero mixed Steklov-Neumann eigenvalue of $M$ with respect to the Riemannian metric $g_{\ve}^\prime = h^{-2k/n} g_{B}^\prime + h^2 g_{F}$ satisfies $\sigma_1(M,g_{\ve}^\prime;\Sigma_{\St} \times F , \Sigma_{\Ne} \times F) \rightarrow + \infty$ as $\ve \rightarrow 0$. It is easily checked that the Riemannian metric $g_{\ve} = h^{-2k/n} g_{B} + h^2 g_{F}$ is quasi-isometric to $g_{\ve}^\prime$ with ratio $C$, and that the volume  element of $g_{\ve}$ coincides with the volume element of the original Riemannian metric $g_B + g_F$ of $M$. We conclude from Proposition \ref{quasi-isometric} that 
$$\sigma_1(M,g_{\ve};\Sigma_{\St} \times F , \Sigma_{\Ne} \times F) \geq \sigma_1(M,g_{\ve}^\prime;\Sigma_{\St} \times F, \Sigma_{\Ne} \times F)/C^{2(n+k)+1} \rightarrow + \infty$$ as $\ve \rightarrow 0$.
\end{proof}

\noindent{\emph{Proof of Theorem \ref{thm:ProductBoundary}.} By assumption, the boundary $\Sigma$ of $M$ splits as the Riemannian product of two closed manifolds $F_1^m$ and $F_2^k$, where $m \geq k\geq 1$. It is evident that there exists $\ell > 0$ such that $\Sigma$ has a neighborhood diffeomorphic to $\Sigma \times [0,2\ell)$ and $ \ell \Vol(\Sigma,g_M)  < \Vol(M,g_M)/2$. Then there exists a Riemannian metric $g$ on $M$ such that $\Sigma$ has a neighborhood $K$ isometric to $\Sigma \times [0,\ell]$ and $\Vol(M,g) = \Vol(M,g_M)$. Indeed, extend the product metric from $\Sigma \times [0,\ell]$ to a Riemannian metric $\bar{g}$ of $M$. Choose $\vf \in C^\infty(M)$ with $\vf = 0$ in $\Sigma \times [0,\ell]$ and $\vf > 0$ in $M \smallsetminus (\Sigma \times [0,\ell])$, and consider the conformal Riemannian metric $\bar{g}_c = e^{c\vf} \bar{g}$, with $c \in \R$. Then the volume of $M$ with respect to $\bar{g}_c$ is given by
\[
\Vol(M,\bar{g}_c) = \int_{M} e^{c(m+k+1)\vf/2}\,dV_{M,\bar{g}}.
\]
It is now evident that $\Vol(M,\bar{g}_c)$ depends continuously on $c \in \R$, while $\Vol(M,\bar{g}_c) \rightarrow + \infty$ as $c \rightarrow +\infty$ and $\Vol(M,\bar{g}_c) \rightarrow \ell \Vol(\Sigma,g_M)  < \Vol(M,g_M)/2$ as $c \rightarrow - \infty$. The intermediate value theorem yields that there exists $c \in \R$ such that $g = \bar{g}_c$ satisfies the asserted properties.

It is clear that $(K,g)$ is the Riemannian product of $B = F_1 \times [0,\ell]$ with $F_2$ and $\dim(B) = m + 1 >\dim(F_2)$. It follows from Theorem \ref{thm:product mixed} that for any $\ve > 0$ there exists a Riemannian metric $g_{\ve}$ on $K$ which coincides with $g$ in a neighborhood of $\partial K = \Sigma \times \{0,\ell\}$ such that $\Vol(K,g_{\ve}) = \Vol(K,g)$ and $\sigma_{1}(K,g_{\ve};\Sigma \times \{0\} , \Sigma \times \{\ell\}) \rightarrow + \infty$ as $\ve \rightarrow 0$. Since $g_{\ve}$ coincides with $g$ in a neighborhood of $\Sigma \times \{\ell\}$ in $K$, we deduce that $g_{\ve}$ extended by $g$ outside $K$ is a Riemannian metric on $M$, which we denote also by $g_{\ve}$, satisfying $\Vol(M,g_{\ve}) = \Vol(M,g) = \Vol(M,g_M)$ and $g_{\ve} = g = g_M$ on $\Sigma$. Finally, we derive from Proposition \ref{diffusion first non-zero eigenvalue} that
\[
\sigma_1(M, g_\ve) = \min_{f} \int_M \| d f \|^2 \geq \min_{f} \int_K \| d f \|^2 = \sigma_1(K,g_\ve; \Sigma \times \{0\} , \Sigma \times \{\ell\}),
\]
where the minimum is taken over all $f \in C^\infty(M)$ satisfying $\int_\Sigma f^2 = 1$ and $\int_{\Sigma} f = 0$. This implies that $\sigma_1(M,g_\ve) \rightarrow + \infty$ as $\ve \rightarrow 0$, as we wished. \qed

\bibliographystyle{plain}
\bibliography{biblio}

\end{document}